\documentclass[a4paper,twoside, reqno]{amsart}

\usepackage{amsmath}
\usepackage{amsthm}
\usepackage{amscd}
\usepackage{amssymb}
\usepackage{array}
\usepackage{hyperref}
\usepackage[pdftex]{graphicx}
\usepackage[latin1]{inputenc}
\usepackage{latexsym}
\usepackage{mathdots}
\usepackage{multirow}
\usepackage{stmaryrd}
\usepackage[all]{xy}

\newtheorem{theorem}{Theorem}[section]
\newtheorem{lemma}[theorem]{Lemma}
\newtheorem{corollary}[theorem]{Corollary}
\newtheorem{definition}[theorem]{Definition}
\newtheorem{proposition}[theorem]{Proposition}
\newtheorem{remark}[theorem]{Remark}

\newcommand{\Ad}[0]{\operatorname{Ad}}
\newcommand{\Aut}[0]{\operatorname{Aut}}

\newcommand{\id}[0]{\operatorname{id}}

\newcommand{\nuc}[0]{\operatorname{nuc}}
\newcommand{\nucdim}[0]{\operatorname{dim}_{\nuc}}

\title{On the nuclear dimension of certain UCT-Kirchberg algebras}
\date{May 20, 2014}

\author{Dominic Enders}
\thanks{This work was supported by the SFB 878 {\itshape Groups, Geometry and Actions} and the Danish National Research Foundation through the {\itshape Centre for Symmetry and Deformation} (DNRF92).}

\address{Dominic Enders
\newline Department of Mathematical Sciences, University of Copenhagen
\newline Universitetsparken 5, DK-2100 Copenhagen \O, Denmark}
\email{d.enders@math.ku.dk}

\begin{document}

\begin{abstract}
We give a partial answer to the question for the precise value of the nuclear dimension of UCT-Kirchberg algebras raised by W. Winter and J. Zacharias in \cite{WZ10}.
It is shown that every Kirchberg algebra in the UCT-class with torsion free $K_1$-group has nuclear dimension 1.
\end{abstract}

\maketitle

\section{Introduction}

In 2010, Winter and Zacharias introduced the notion of nuclear dimension for \linebreak $C^*$-algebras as a non-commutative analogue of topological covering dimension \linebreak (\cite{WZ10}).
Since then it has clearly become one of the most important concepts \linebreak in the Elliott classification program, i.e. the classification of simple, nuclear \linebreak $C^*$-algebras by $K$-theoretic data.
In particular, finite-dimensionality with respect to this notion of dimension constitutes one of the fundamental regularity properties for the $C^*$-algebras in question.
It is part of the Toms-Winter conjecture that in this situation finite nuclear dimension is presumably equivalent to $\mathcal{Z}$-stability and therefore a necessary assumption in order to obtain classification by the Elliott invariant.
This conjecture received a great deal of attention and the equivalence mentioned above has been partially verified (\cite{SWW14}, \cite{Win12}).
On the other hand, large classes of simple, nuclear $C^*$-algebras with finite nuclear dimension have been successfully classified by their Elliott invariant (e.g. \cite{Lin11}, \cite{Win12}).

However, once a $C^*$-algebra having finite nuclear dimension belongs to a class which is classified by some invariant, it is a natural question to ask for the precise value of its dimension and how it can be read off from its invariant.
In this paper we study this question for the class of UCT-Kirchberg algebras, i.e. purely infinite, simple, separable, nuclear $C^*$-algebras satisfying the UCT which are, due to Kirchberg and Phillips, completely classified by their $K$-groups.
In \cite{WZ10}, Winter and Zacharias showed that all these algebras have nuclear dimension at most 5 and asked whether the precise value of their dimension is determined by algebraic properties of their $K$-groups, such as torsion (\cite[Problem 9.2]{WZ10}).
Here, we give a partial answer to their question and give the optimal dimension estimate in the absence of $K_1$-torsion.
More precisely, we show that the nuclear dimension of UCT-Kirchberg algebras with torsion free $K_1$-groups equals 1 (Theorem \ref{main}).

Our methods of proof used in the computation of the nuclear dimension differ from the arguments in the original estimate by Winter and Zacharias in two essential ways.
First, we do not make use of Cuntz-Pimsner algebras, as done in \cite{WZ10}, but employ crossed product models for UCT-Kirchberg algebras which are provided by the work of R{\o}rdam (\cite{Ror95}).
Second, we do not construct completely positive approximations for the identity map on a given Kirchberg algebra $A$ directly.
Instead, we study a certain family of homomorphisms $\iota_n\colon A\rightarrow M_n(A)$ and show that the $\iota_n$ can be approximated in a decomposable manner.
While the error we have to make in the approximation of each $\iota_n$ is bounded from below, the lower bound will be small for large $n$.
In a second step, we show how to combine the $\iota_n$ in order to construct a homomorphism $A\rightarrow M_k(A)$ which admits decomposable approximations of the same quality as the $\iota_n$, but in addition induces an isomorphism on $K$-theory.  
By Kirchberg-Phillips classification, this map can be perturbed to the identity map on $A$ and by that provides suitable approximations for estimating the nuclear dimension of $A$. \vspace{1cm}

This paper is organized as follows: First, right now, we define certain matrix embeddings for crossed products which are the main objects of study in this paper.
\begin{definition}\label{iota}
Let $A$ be a $C^*$-algebra, $\alpha\in\Aut(A)$ an automorphism and $A\rtimes_\alpha\mathbb{Z}$ the corresponding full crossed product.
For $n\in\mathbb{N}$, $n\geq 2$, we let
\[\iota_n\colon A\rtimes_\alpha\mathbb{Z} \rightarrow M_n(A\rtimes_\alpha\mathbb{Z})\]
denote the embedding given by 
\[\begin{tabular}{ccc}$a \mapsto\begin{pmatrix}
                  \alpha^{-1}(a) & 0 & \cdots& 0 \\  0 & \alpha^{-2}(a) & \ddots& \vdots\\ \vdots & \ddots& \ddots& 0 \\  0 & \cdots& 0 &  \alpha^{-n}(a)
                 \end{pmatrix}$
& and
& $U \mapsto\begin{pmatrix}
                  0 & \cdots & \cdots & 0 & U^n \\
                  1 & \ddots &&& 0 \\
                  0 & \ddots & \ddots & & \vdots \\
                  \vdots & \ddots & \ddots & \ddots & \vdots \\
                  0 & \cdots & 0 & 1 &0
                 \end{pmatrix}$\end{tabular}\]
for $a\in A$ resp. for the unitary $U\in\mathcal{M}(A\rtimes_\alpha\mathbb{Z})$ implementing the action $\alpha$.
\end{definition}
In section \ref{section approximation}, we show that these maps $\iota_n$ admit completely positive approximations in an order zero fashion, at least up to an error which is small for large $n$. 
The $K$-theory of these maps can be computed in many cases of interest, this is done in section \ref{section k-theory}.
A combination of these results yields the dimension estimate for Kirchberg algebras in section \ref{section main}. 
We finish with some remarks about possible generalizations of the technique developed in this paper.


\section{Decomposable approximations for $\iota_n$}\label{section approximation}

We want to study certain maps which admit completely positive approximations of the same type as in the original definition of nuclear dimension in \cite{WZ10}.
In our situation, however, we won't be able to find approximations of arbitrary small error (which would rather lead to a notion of nuclear dimension for maps).
Therefore we need to keep track of the precision of approximation that we can get for a given map.
This will be done using the following definition.

\begin{definition}\label{def decomposition}
Let $\alpha\colon A\rightarrow B$ be a map between $C^*$-algebras $A$ and $B$.
Given \linebreak $d\in\mathbb{N}_0$, $\epsilon>0$ and a finite subset $\mathcal{G}$ of $A$, we say that $\alpha$ admits a piecewise contractive, completely positive, $d$-decomposable $(\mathcal{G},\epsilon)$-approximation if
there exists $(F,\psi,\varphi)$ such that $F$ is a finite-dimensional $C^*$-algebra, and such that $\psi\colon A\rightarrow F$ and $\varphi\colon F\rightarrow B$ are completely positive maps satisfying

\begin{enumerate}
 \item $\|(\varphi\circ\psi)(a)-\alpha(a)\|<\epsilon$ for all $a\in\mathcal{G}$;
 \item $\psi$ is contractive;
 \item $F$ decomposes as $F=F^{(0)}\oplus\cdots\oplus F^{(d)}$ such that $\varphi_{|F^{(i)}}$ is a c.p.c. order zero map for each $i=0,...,d$.
\end{enumerate}
\end{definition}

In this terminology, the notion of nuclear dimension for a $C^*$-algebra $A$, as defined by Winter and Zacharias in \cite{WZ10}, can be formulated as follows:

\[\begin{tabular}{rc}
\multirow{2}{*} {$\nucdim(A)\leq d\quad\Leftrightarrow$}  & 
$\id_A$ admits piecewise contractive, c.p., $d$-decomposable, \\ & $(\mathcal{G},\epsilon)$-approximations for all finite $\mathcal{G}\subset A$ and all $\epsilon>0$. 
\end{tabular}\]

We first note that the approximations of \ref{def decomposition} for $^*$-homomorphisms are well-behaved with respect to composition, orthogonal sum and approximate unitary equivalence.
We collect some of these elementary permanence results in the following lemma.

\begin{lemma}\label{permanence}
Let $*$-homomorphisms $\xymatrix{A \ar[r]^\alpha & B \ar[r]^{\beta_1,\beta_2} & C\ar[r]^\gamma & D}$ between \linebreak $C^*$-algebras $A,B,C,D$, a finite subset $\mathcal{G}\subset B$, $\epsilon>0$ and $d\in\mathbb{N}_0$ be given.
If both $\beta_1$ and $\beta_2$ admit piecewise contractive, completely positive, $d$-decomposable $(\mathcal{G},\epsilon)$-approximations $(F_i,\psi_i,\varphi_i)$, $i\in\{1,2\}$, then the following holds:
\begin{enumerate}
 \item The composition $\gamma\circ\beta_1$ admits a piecewise contractive, completely positive, $d$-decomposable $(\mathcal{G},\epsilon)$-approximation.
 \item For any $\alpha$-preimage $\mathcal{G}'$ of $\mathcal{G}$, the composition $\beta_1\circ\alpha$ admits a piecewise contractive, completely positive, $d$-decomposable $(\mathcal{G'},\epsilon)$-approximation.
 \item The orthogonal sum $\begin{pmatrix}\beta_1 & 0 \\ 0 & \beta_2\end{pmatrix}\colon B\rightarrow M_2(C)$ admits a piecewise contractive, completely positive, $d$-decomposable $(\mathcal{G},\epsilon)$-approximation.
 \item If $\beta_1\sim_{a.u.}\delta$ for some $\delta\colon B\rightarrow C$, then $\delta$ admits a piecewise contractive, completely positive, $d$-decomposable $(\mathcal{G},\epsilon)$-approximation.
\end{enumerate}
\end{lemma}

\begin{proof}
We only name suitable approximating systems and leave the details to the reader.
For (1) one chooses $(F_1,\psi_1,\gamma\circ\varphi_1)$, while for (2) $(F_1,\psi_1\circ\alpha,\varphi_1)$ works.
For (3) consider $\left(F_1\oplus F_2,\psi_1\oplus\psi_2,\begin{pmatrix}\varphi_1 & 0 \\ 0 & \varphi_2\end{pmatrix}\right)$ with respect to the decomposition $(F_1\oplus F_2)^{(i)}=F_1^{(i)}\oplus F_2^{(i)}$.
Given $\delta=\lim\limits_{n\rightarrow\infty}\Ad(U_n)\circ\beta_1$ in (4), one checks that the triple $(F,\psi_1,\Ad(U_n)\circ\varphi_1)$ works for $n$ sufficiently large.
\end{proof}

Next, we apply a 'cut and paste'-technique similar to the one used in \cite[section 7]{WZ10} to the maps $\iota_n\colon A\rtimes_\alpha\mathbb{Z}\rightarrow M_n(A\rtimes_\alpha\mathbb{Z})$ of Definition \ref{iota}.
This essentially reduces a decomposable approximation of $\iota_n$ to a decomposable approximation of the coefficient algebra $A$.

\begin{lemma}\label{approximation1}
 Let a nuclear $C^*$-algebra $A$, an automorphism $\alpha\in\Aut(A)$, a finite subset $\mathcal{G}\subset A\rtimes_\alpha\mathbb{Z}$ and $\epsilon>0$ be given.
 Then the following holds for the maps $\iota_n$ as defined in \ref{iota}: 
 For all sufficiently large $n$ there exists a c.c.p. map $\psi_n\colon A\rtimes_\alpha\mathbb{Z}\rightarrow M_n(A)$ and two $\ast$-homomorphisms $\Lambda_n^0,\Lambda_n^1\colon M_n(A)\rightarrow M_n(A\rtimes_\alpha\mathbb{Z})$ such that the diagram
\[
 \begin{xy}
  \xymatrix{A\rtimes_\alpha\mathbb{Z} \ar[rr]^{\iota_n} \ar[dr]_{\psi_n} && M_n(A\rtimes_\alpha\mathbb{Z}) \\ & M_n(A) \ar[ur]_{\Lambda_n^0+\Lambda_n^1}}
 \end{xy}
\]
 commutes up to $\epsilon$ on $\mathcal{G}$, i.e. $\|\iota_n(x)-((\Lambda_n^0+\Lambda_n^1)\circ\psi_n)(x)\|<\epsilon$ holds for all $x\in\mathcal{G}$.
\end{lemma}

\begin{proof}
Assume $A\subseteq\mathcal{B}(H)$ for some Hilbert space $H$.
By nuclearity of $A$ we have $A\rtimes_\alpha\mathbb{Z}=A\rtimes_{\alpha,r}\mathbb{Z}$ (\cite[Theorem 4.2.4]{BO08}) and may therefore identify the crossed product with the $C^*$-subalgebra of $\mathcal{B}(\ell^2(\mathbb{Z})\otimes H)$ generated by the operators

\[\begin{array}{rclr}
   a(\delta_i\otimes\xi)&=&\delta_i\otimes(\alpha^{-i}(a))(\xi),&a\in A \\
   U(\delta_i\otimes\xi)&=&\delta_{i+1}\otimes\xi
  \end{array}\]
respectively by the ideal therein generated by $A$ in the non-unital case.

Denote by $P_n$ the projection onto the subspace $\ell^2(\{1,...,n\})\otimes H$.
Compression by $P_n$ gives a c.p.c. map from $A\rtimes_\alpha\mathbb{Z}$ to $M_n(A)$.
We need to put suitable weights on the entries of the elements of $\{P_nxP_n\}_{x\in\mathcal{G}}$. 
As in \cite{WZ10}, this will be done by Schur multiplication (i.e. entrywise multiplication) with a suitable positive contraction $\kappa_n\in M_n(\mathbb{C})$.
We choose 

\[\kappa_n=\frac{1}{m+1}\begin{pmatrix}
1 & 1 & \cdots & 1 & 1 & 1 & 1 & \cdots & 1 & 1 \\
1 & 2 & \cdots & 2 & 2 & 2 & 2 & \cdots & 2 & 1 \\
\vdots & \vdots & \ddots & \vdots & \vdots & \vdots & \vdots & \iddots & \vdots & \vdots \\
1 & 2 & \cdots & m-1 & m-1 & m-1 & m-1 & \cdots & 2 & 1 \\
1 & 2 & \cdots & m-1 & m & m & m-1 & \cdots & 2 & 1 \\
1 & 2 & \cdots & m-1 & m & m & m-1 & \cdots & 2 & 1 \\
1 & 2 & \cdots & m-1 & m-1 & m-1 & m-1 & \cdots & 2 & 1 \\
\vdots & \vdots & \iddots & \vdots & \vdots & \vdots & \vdots & \ddots & \vdots & \vdots \\
1 & 2 & \cdots & 2 & 2 & 2 & 2 & \cdots & 2 & 1 \\
1 & 1 & \cdots & 1 & 1 & 1 & 1 & \cdots & 1 & 1 
\end{pmatrix}\]
if $n=2m$ is even, respectively

\[\kappa_n=\frac{1}{m+1}\begin{pmatrix}
1 & 1 & \cdots & 1 & 1 & 1 & \cdots & 1 & 1 \\
1 & 2 & \cdots & 2 & 2 & 2 & \cdots & 2 & 1 \\
\vdots & \vdots & \ddots & \vdots & \vdots & \vdots & \iddots & \vdots & \vdots \\
1 & 2 & \cdots & m & m & m & \cdots & 2 & 1 \\
1 & 2 & \cdots & m & m+1 & m & \cdots & 2 & 1 \\
1 & 2 & \cdots & m & m & m & \cdots & 2 & 1 \\
\vdots & \vdots & \iddots & \vdots & \vdots & \vdots & \ddots & \vdots & \vdots \\
1 & 2 & \cdots & 2 & 2 & 2 & \cdots & 2 & 1 \\
1 & 1 & \cdots & 1 & 1 & 1 & \cdots & 1 & 1 
\end{pmatrix}\]
if $n=2m+1$ is odd.
It is not hard to see that the map on $M_n(A)$ given by $y\mapsto\kappa_n\ast y$, where $\ast$ denotes Schur multiplication, is then in fact completely positive and contractive. 
Therefore, $\psi_n$ defined as
\[\begin{array}{cc}\psi_n\colon A\rtimes_\alpha\mathbb{Z}\rightarrow M_n(A),&x\mapsto\kappa_n\ast(P_nxP_n)\end{array}\]
is also a c.p.c. map.

Now let $j\colon M_n(A)\rightarrow M_n(A\rtimes_\alpha\mathbb{Z})$ be the canonical embedding. Using the unitary element

\[V:=\begin{pmatrix}
      0 & \cdots & \cdots & 0 & U^{n-1} \\
                  U^{-1} & \ddots &&& 0 \\
                  0 & \ddots & \ddots & & \vdots \\
                  \vdots & \ddots & \ddots & \ddots & \vdots \\
                  0 & \cdots & 0 & U^{-1} &0
     \end{pmatrix}\in M_n(\mathcal{M}(A\rtimes_\alpha\mathbb{Z})),\] 
we consider the $^*$-homomorphisms $\Lambda_n^0,\Lambda_n^1\colon M_n(A)\rightarrow M_n(A\rtimes_\alpha\mathbb{Z})$ given by $\Lambda_n^0=j$ and $\Lambda_n^1=Ad\left(V^m\right)\circ j$.\\

We may assume that the elements of $\mathcal{G}$ are of the form $x=a_xU^{k_x}$ for suitable contractions $a_x\in A$ and $k_x\in\mathbb{Z}$, let $K:=\max_{x\in\mathcal{G}}|k_x|$.
We claim that, for $n>\frac{K}{\epsilon}$, the maps $\Lambda_n^0,\Lambda_n^1$ and $\psi_n$ defined above provide an approximation of $\iota_n$ as claimed.
For convenience we only consider the case of even $n$, the calculations for the odd case are essentially identical.
Given $x=aU^{k}\in\mathcal{G}$ with $k\geq 0$, one checks that $\iota_n(x)-((\Lambda_n^0+\Lambda_n^1)\circ\psi_n)(x)$ equals
\[\frac{1}{(\frac{n}{2}+1)}\begin{pmatrix}
0&0&0&k\cdot I_k-\mu_k \\ k\cdot I_{(\frac{n}{2}-k)} &0&0&0\\ 0& k\cdot I_k - \mu_k &0&0 \\0&0&k\cdot I_{(\frac{n}{2}-k)}&0
\end{pmatrix}\ast\iota_n(x)\]
where 
\[\mu_k=\begin{pmatrix}
0 & 0 & \cdots &&& \cdots & 0 \\
0 & 1 & 0 &&&& \vdots \\
\vdots & 0 & 2 & \ddots \\
&&\ddots & \ddots & \ddots \\
&&&\ddots & 2 & 0 & \vdots \\
\vdots &&&&0&1&0\\
0&\cdots&&&\cdots&0&0\end{pmatrix}\in M_k(\mathbb{C}),\]
$I_d$ is the $d\times d$ unit matrix and $\ast$ denotes again Schur multiplication. 
Due to the special off-diagonal form of these elements, an entrywise norm estimate immediately shows
\[\|\iota_n(x)-((\Lambda_n^0+\Lambda_n^1)\circ\psi_n)(x)\|\leq\frac{k}{(\frac{n}{2}+1)}\|x\|\leq\frac{K}{n}<\epsilon.\]
\end{proof}

Next, we show that the maps $\iota_n$ of \ref{iota} admit decomposable approximations in the sense of Definition \ref{def decomposition} provided that the coefficient algebra has finite nuclear dimension.

\begin{proposition}\label{approximation2}
 Given a $C^*$-algebra $A$ with $\nucdim(A)=d<\infty$, a finite subset $\mathcal{G}\subset A$ and $\epsilon>0$, the maps $\iota_n\colon A\rtimes_\alpha\mathbb{Z}\rightarrow M_n(A\rtimes_\alpha\mathbb{Z})$ defined in \ref{iota} admit a piecewise contractive, completely positive, $(2d+1)$-decomposable $(\mathcal{G},\epsilon)$-approximation for all sufficiently large $n$. 
\end{proposition}

\begin{proof}
 Choose $n$ large enough such that there exist maps $\psi_n$ and $\Lambda_n^0,\Lambda_n^1$ as in Lemma \ref{approximation1} satisfying $\|\iota_n(x)-\left((\Lambda_n^0+\Lambda_n^1)\circ\psi_n\right)(x)\|<\frac{\epsilon}{3}$ for all $x\in\mathcal{G}$.
 Since $\nucdim(M_n(A))=\nucdim(A)=d$ by \cite[Corollary 2.8]{WZ10}, there exists a piecewise contractive, completely positive, $d$-decomposable $(\psi_n(\mathcal{G}),\frac{\epsilon}{3})$-approximation of $\id_{M_n(A)}$,
 i.e. a finite-dimensional $C^*$-algebra $F=F^{(0)}\oplus...\oplus F^{(d)}$, a c.p.c. map $\psi'\colon M_n(A)\rightarrow F$ and c.p.c. order zero maps $\varphi'^{(i)}\colon F^{(i)}\rightarrow M_n(A)$, $i=0,...,d$, such that 
 \[\left\|\left(\sum_{i=0}^d\varphi'^{(i)}\circ\psi'\right)(\psi_n(x))-\psi_n(x)\right\|<\frac{\epsilon}{3}\] holds for all $x\in\mathcal{G}$.
 Putting these two approximations together, i.e. setting $G:=F\oplus F$ with decomposition $G=\bigoplus_{j=0,1}\bigoplus_{i=0}^d G^{(i,j)}$, where $G^{(i,j)}:=F^{(i)}$,
and considering
 \[\begin{xy}\xymatrix{
 A\rtimes_\alpha\mathbb{Z} \ar[rr]^{\iota_n} \ar[dr]_(0.4){(\psi'\circ\psi_n)\oplus(\psi'\circ\psi_n)\;\;\;\;\;} && M_n(A\rtimes_\alpha\mathbb{Z}) \\
& G\oplus G=\bigoplus\limits_{j=0,1}\;\bigoplus\limits_{i=0}^d G^{(i,j)} \ar[ur]_(0.7){\sum\limits_{j=0,1}\;\sum\limits_{i=0}^d\left(\Lambda_n^j\circ\varphi'^{(i)}\right)}
 }\end{xy},\]
 one finds each $\Lambda_n^j\circ\varphi'^{(i)}$ to be an order zero map and further
 \[\begin{array}{rl}&\left\|\sum\limits_{j=0}^1\sum\limits_{i=0}^d\left(\Lambda_n^j\circ\varphi'^{(i)}\right)((\psi'\circ\psi_n)(x))-\iota_n(x)\right\| \\ \\
    \leq & \left\|\sum\limits_{j=0}^1\Lambda_n^j(\psi_n(x))-\iota_n(x)\right\| +2\left\|\left(\sum\limits_{i=0}^d\varphi'^{(i)}\circ\psi'\right)(\psi_n(x))-\psi_n(x)\right\| \\ \\ < & \epsilon
\end{array}\]
 for all $x\in\mathcal{G}$.  
 In other words, $\left(G,(\psi'\circ\psi_n)^{\oplus 2},\sum_{i,j}\left(\Lambda_n^j\circ\varphi'^{(i)}\right)\right)$ provides a piecewise contractive, completely positive, $(2d+1)$-decomposable $(\mathcal{G},\epsilon)$-approximation for $\iota_n$.
\end{proof}

\section{The $K$-theory of $\iota_n$}\label{section k-theory}

In this section we compute the map on $K$-theory induced by $\iota_n\colon A\rtimes_\alpha\mathbb{Z}\rightarrow M_n(A\rtimes_\alpha\mathbb{Z})$ from \ref{iota}.
Using a rotation argument we first express $K_*(\iota_n)$ in terms of $K_*(\alpha)$. 
Under suitable assumptions, this allows us to read off $K_*(\iota_n)$ from the Pimsner-Voiculescu sequence associated to $(A,\alpha)$.

\begin{lemma}\label{homotopy}
 The homomorphism $\iota_n\colon A\rtimes_\alpha\mathbb{Z}\rightarrow M_n(A\rtimes_\alpha\mathbb{Z})$ defined in \ref{iota} is homotopic to the diagonal embedding $j_n\colon A\rtimes_\alpha\mathbb{Z}\rightarrow M_n(A\rtimes_\alpha\mathbb{Z})$ given by
 \[\begin{tabular}{ccc}$a \mapsto\begin{pmatrix}
                  \alpha^{-1}(a) & 0 & \cdots& 0 \\  0 & \alpha^{-2}(a) & \ddots& \vdots\\ \vdots & \ddots& \ddots& 0 \\  0 & \cdots& 0 &  \alpha^{-n}(a)
                 \end{pmatrix}$
& and
& $U \mapsto\begin{pmatrix}
                  U \\ & U \\ &&\ddots \\&&&U
                 \end{pmatrix}$\end{tabular}\]
for $a\in A$ resp. for the unitary $U\in\mathcal{M}(A\rtimes_\alpha\mathbb{Z})$ implementing the action $\alpha$.
\end{lemma}

\begin{proof}
 Let $(V_t)_{t\in[0,1]}$ be a path of unitaries in $M_n(\mathbb{C})$ connecting the shift-unitary $V_0=\sum_{i=1}^{n-1} e_{i,i+1}+e_{n,1}$ to identity matrix $V_1=1_n$ and consider the path $(\kappa_t)_{t\in[0,1]}$ of homomorphisms $\kappa_t\colon A\rtimes_\alpha\mathbb{Z}\rightarrow M_n(A\rtimes_\alpha\mathbb{Z})$ given by
  
  \[\begin{tabular}{ccc}$a \mapsto\begin{pmatrix}
                  \alpha^{-1}(a) & 0 & \cdots& 0 \\  0 & \alpha^{-2}(a) & \ddots& \vdots\\ \vdots & \ddots& \ddots& 0 \\  0 & \cdots& 0 &  \alpha^{-n}(a)
                 \end{pmatrix}$
& and
& $U \mapsto V_t\ast\begin{pmatrix}
                  U & U^2 & \cdots & U^n \\
                  1 & U & \ddots & \vdots \\
                  \vdots & \ddots & \ddots & U^2 \\
                  U^{-n+2} & \cdots & 1 & U
                 \end{pmatrix}$\end{tabular}\]
where $\ast$ denotes Schur multiplication.
It is straightforward to check that $\kappa_t(U)$ is a continuous unitary path in $M_n(\mathcal{M}(A\rtimes_\alpha\mathbb{Z}))$, hence we only have to make sure that each $\kappa_t$ is well-defined, i.e. compatible with the action $\alpha$:
\[\begin{array}{rl}
   &\left(\kappa_t(U)\kappa_t(a)\kappa_t(U)^*\right)_{ij} \\
   = & \sum_{k=1}^n \kappa_t(U)_{ik}\alpha^{-k}(a)\kappa_t(U^*)_{kj} \\
   = & \sum_{k=1}^n (V_t)_{ik}\overline{(V_t)_{jk}}U^{(k-i+1)}\alpha^{-k}(a)U^{-(k-j+1)} \\
   = & \left(\sum_{k=1}^n (V_t)_{ik}\overline{(V_t)_{jk}}\right) U^{-i+1}aU^{j-1} \\
   = & \delta_{ij}\cdot U^{-i+1}aU^{j-1} \\
   = & \delta_{ij}\cdot \alpha^{-i+1}(a) \\
   = & \left(\kappa_t(\alpha(a))\right)_{ij}
  \end{array}\]
This shows that $\kappa_t$ provides a homotopy between $\kappa_0=\iota_n$ and $\kappa_1=j_n$.
\end{proof}

\begin{proposition}\label{k-theory}
Given $n\in\mathbb{N}$, let $\iota_n\colon A\rtimes_\alpha\mathbb{Z}\rightarrow M_n(A\rtimes_\alpha\mathbb{Z})$ be the homomorphism defined in \ref{iota}.
If both boundary maps in the Pimsner-Voiculescu sequence associated to $(A,\alpha)$ vanish, we find $K_*(\iota_n)=n\cdot \id_{K_*(A\rtimes_\alpha\mathbb{Z})}$.
The same conclusion holds if instead either $K_0(A)=0$ or $K_1(A)=0$.
\end{proposition}

\begin{proof}
We know by Lemma \ref{homotopy} that $(\iota_n)_*=\sum_{i=1}^n (\widehat{\alpha^i})_*$ where $\hat{\beta}$ denotes the canonical extension of an automorphism $\beta\in\Aut(A)$ satisfying $\alpha\circ\beta=\beta\circ\alpha$ to an automorphism of $A\rtimes_\alpha\mathbb{Z}$ via $\widehat{\beta}(U)=U$. 
Using $(\widehat{\alpha^i})_*=(\widehat{\alpha}_*)^i$, it suffices to show that $\widehat{\alpha}$ induces the identity map on $K$-theory.
By the Pimsner-Voiculescu sequence (\cite[Theorem 10.2.1]{Bla98}) and the naturality thereof, we have the following commutative diagram with exact rows
\[\begin{xy}\xymatrix{
\ar[r] & K_*(A) \ar[r]^{1-\alpha_*} \ar[d]^{\alpha_*} & K_*(A) \ar[r]^(.4){i_*} \ar[d]^{\alpha_*} & K_*(A\rtimes_\alpha\mathbb{Z}) \ar[r]^{\partial_*} \ar[d]^{\widehat{\alpha}_*} & K_{*+1}(A) \ar[r]^{1-\alpha_{*+1}} \ar[d]^{\alpha_{*+1}} & K_{*+1}(A) \ar[r] \ar[d]^{\alpha_{*+1}} & \\
\ar[r] & K_*(A) \ar[r]^{1-\alpha_*} & K_*(A) \ar[r]^(.4){i_*} & K_*(A\rtimes_\alpha\mathbb{Z}) \ar[r]^{\partial_*} & K_{*+1}(A) \ar[r]^{1-\alpha_{*+1}} & K_{*+1}(A) \ar[r] & }\end{xy}\]
where $i$ denotes the canonical inclusion of $A$ into the crossed product.

In the first case, i.e. $\partial_*=0$, we have an isomorphism $K_*(A)/(1-\alpha_*)K_*(A)\cong K_*(A\rtimes_\alpha\mathbb{Z})$ induced by $i_*$. 
Hence $\widehat{\alpha}_*\circ i_*=i_*\circ\alpha_*=i_*$ and the claim follows by surjectivity of $i_*$.

The second case is treated similarly. If $K_i(A)=0$, then $K_{i+1}(A\rtimes_\alpha\mathbb{Z})\cong K_{i+1}(A)/(1-\alpha_{i+1})K_{i+1}(A)$ and $\widehat{\alpha}_{i+1}=\id$ follows as in the first case.
On the other hand, the boundary map $\partial_i$ identifies $K_i(A\rtimes_\alpha\mathbb{Z})$ with $ker(1-\alpha_{i+1})\subseteq K_{i+1}(A)$. 
Hence $\partial_i\circ\widehat{\alpha}_i=\alpha_{i+1}\circ\partial_i=\partial_i$, which shows that also $\widehat{\alpha}_i=\id$.
\end{proof}

\section{Nuclear dimension of Kirchberg algebras}\label{section main}
Combining the results of section \ref{section approximation} and \ref{section k-theory} with the $K$-theoretical classification results by Kirchberg and Phillips, we obtain our main result, the computation of the nuclear dimension for UCT-Kirchberg algebras in the absence of torsion in $K_1$.

\begin{theorem}\label{main}
 Let $B$ be a Kirchberg algebra in the UCT-class. If $K_1(B)$ is torsion free, then $\nucdim(B)=1$.
\end{theorem}

\begin{proof}
Since $\nucdim(B)=\nucdim(B\otimes\mathbb{K})$ by \cite[Corollary 2.8]{WZ10}, we may assume that $B$ is stable.
In this case, $B$ can be realized as a crossed product $B=A\rtimes_\alpha\mathbb{Z}$ with $A$ an AF-algebra by \cite[Corollary 8.4.11]{Ror02}.
This implies that, given any finite subset $\mathcal{G}$ of $B=A\rtimes_\alpha\mathbb{Z}$ and $\epsilon>0$, we can use Proposition \ref{approximation2} to find a natural number $n$ such that both $\iota_n\colon B\rightarrow M_n(B)$ and $\iota_{n+1}\colon B\rightarrow M_{n+1}(B)$ (as defined in \ref{iota}) admit piecewise contractive, completely positive, 1-decomposable $(\mathcal{G},\epsilon)$-approximations $(F_i,\psi_i,\varphi_i)$, $i\in\{n,n+1\}$. 
Using the existence part of Kirchberg-Phillips classification (\cite[Theorem 8.4.1]{Ror02}), we can choose an automorphism $\omega$ on $M_n(B)$ which induces $\omega_*=-\id$ on $K$-theory.
Now consider the embedding $\iota$ given by
\[\iota=\begin{pmatrix}\iota_{n+1} & 0 \\ 0 & \omega\circ\iota_n\end{pmatrix}\colon B\rightarrow M_{2n+1}(B).\]
By Lemma \ref{permanence}, $\iota$ also admits a piecewise contractive, completely positive, 1-decom\-posable $(\mathcal{G},\epsilon)$-approximation.
By Proposition \ref{k-theory} and the choice of $\omega$ we further have
\[K_*(\iota)=K_*(\iota_{n+1})+K_*(\omega)\circ K_*(\iota_n)=(n+1)\cdot\id-n\cdot\id=\id.\]
Since $B$ satisfies the UCT, this shows that the class of $\iota$ is invertible in \linebreak $KK(B,M_{2n+1}(B))$ and its inverse is, using \cite[Theorem 8.4.1]{Ror02} again, induced by a $*$-isomomorphism $\varrho\colon M_{2n+1}(B)\rightarrow B$.
The uniqueness part of Kirchberg-Phillips classification (\cite[Theorem 8.4.1]{Ror02}) implies that $\varrho\circ\iota$ is approximately unitarily equivalent to $\id_B$.
Hence, by Lemma \ref{permanence}, $\id_B$ also admits a piecewise contractive, completely positive, 1-decomposable $(\mathcal{G},\epsilon)$-approximation. 

Since $\mathcal{G}$ and $\epsilon$ were arbitrary, this shows $\nucdim(B)\leq 1$.
Because $B$ is not an AF-algebra, the nuclear dimension of $B$ must in fact be equal to 1. 
\end{proof}

For the general case, i.e. for Kirchberg algebras without any $K$-theory constraints, we immediately get the following.
Note that the same estimate has by now also be obtained in \cite[Theorem 7.1]{MS13} and \cite[Corollary 3.4]{BEMSW14} by methods completely different from the one developed here. 
These results are more general though, they do not require the UCT.

\begin{corollary}
 A Kirchberg algebra in the UCT-class has nuclear dimension at most 3.
\end{corollary}

\begin{proof}
This is proven exactly as in \cite[Theorem 7.5]{WZ10} except for using the improved estimate $\nucdim(\mathcal{O}_\infty)=1$ which is given by Theorem \ref{main}.
\end{proof} 

The strategy of this paper can be used to obtain dimension estimates in more general situations (cf. Remark \ref{Efren} below).
In fact, there are only two ingredients needed:
First, one needs sufficient classfication results, i.e. both existence and uniqueness results for homomorphisms between $C^*$-algebras in the class under consideration.
Second, one requires the existence of a family of homomorphisms which admit good approximations in a decomposable manner.
These homomorphisms should further induce maps on the level of invariants which can be used to build up an invertible element.
In that case, one can proceed as in Theorem \ref{main} to obtain dimension estimates.
However, the proof of \ref{main} seems to be limited to infinite $C^*$-algebras. 
This is because we require an automorphism $\omega$ which gives a sign on $K$-theory, i.e. which satisfies $\omega_*=-\id$.
However, the map $-\id$ will not be an automorphism as soon as the invariant involves some non-trivial order structure on the $K$-groups.
Therefore one cannot find such $\omega$ in this case.

\begin{remark}\label{Efren}
After having recieved a preprint of this paper, Ruiz, Sims and Tomforde estimated the nuclear dimension of certain graph $C^*$-algebras in \cite{RST13}, following the strategy outlined above.
Their result includes some, but not all Kirchberg algebras covered by Theorem \ref{main}.
Moreover, using classification results involving filtered-$K$-theory they also obtain estimates in the non-simple case.
\end{remark}

\begin{remark}\label{AT-case}
In the presence of torsion in $K_1$, one can still represent stable UCT-Kirchberg algebras $B$ as crossed products $A\rtimes_\alpha\mathbb{Z}$. 
We can no longer choose $A$ to be an AF-algebra, but one can settle for $A$ to be an A$\mathbb{T}$-algebra instead.
Furthermore, the crossed product can be constructed in such a way that the boundary maps in the associated Pimsner-Voiculescu sequence vanish (Theorem 3.6 of \cite{Ror95} and its proof) and therefore Proposition \ref{k-theory} applies. 
Following the lines of \ref{main} and using nuclear one-dimensionality of A$\mathbb{T}$-algebras, this gives an alternative proof for the general estimate $\nucdim(B)\leq 3$.
However, in this case it looks like the maps $\iota_n$ from \ref{iota} even admit $2$-decomposable approximations when contructed more carefully. 
This would give an improved dimension estimate in the torsion case.
\end{remark}

\end{document}